\documentclass[a4paper,11pt]{amsart}
\usepackage[latin1]{inputenc}
\usepackage[T1]{fontenc}
\usepackage{amssymb,amsfonts,amsmath,amsthm}
\usepackage{color}
\usepackage{setspace}
\usepackage{paralist}
\usepackage{epsfig,graphicx,psfrag}
\usepackage[colorlinks=true, linkcolor=blue, citecolor=red,
urlcolor=blue]{hyperref}

\newcounter{notes}
\newcommand{\ignore}[1]{}

\newtheorem{theorem}{Theorem}
\newtheorem{proposition}[theorem]{Proposition}
\newtheorem{corollary}[theorem]{Corollary}
\newtheorem{lemma}[theorem]{Lemma}

\newtheorem{question}[theorem]{Question}

\theoremstyle{definition}
\newtheorem{definition}[theorem]{Definition}

\theoremstyle{remark}
\newtheorem{remark}[theorem]{Remark}

\newtheoremstyle{theoremwithref}{}{}{\itshape}{}{\bfseries}{.}{.5em}{#1 #2 #3}
\theoremstyle{theoremwithref}

\newcommand{\PP}{\mathbf{P}}
\newcommand{\CC}{\mathbf{C}}

\newcommand{\RR}{\mathbf{R}}

\newcommand{\KK}{\mathbf{K}}
\newcommand{\GL}{\mathrm{GL}}
\newcommand{\SL}{\mathrm{SL}}
\newcommand{\PSL}{\mathrm{PSL}}
\newcommand{\SO}{\mathrm{SO}}

\newcommand{\U}{\mathrm{U}}
\newcommand{\Sp}{\mathrm{Sp}}

\newcommand{\id}{\mathrm{Id}}

\begin{document}

\title{Representations and geometric structures}

\author[A. Wienhard]{Anna Wienhard}
\address{Ruprecht-Karls Universit\"at Heidelberg, Mathematisches Institut, Im Neuenheimer Feld~288, 69120 Heidelberg, Germany
\newline HITS gGmbH, Heidelberg Institute for Theoretical Studies, Schloss-Wolfs\-brunnen\-weg 35, 69118 Heidelberg, Germany}
\email{wienhard@mathi.uni-heidelberg.de}

\thanks{The author was partially supported by the National Science Foundation under agreements DMS-1065919 and 0846408, by the Sloan Foundation, by the Deutsche Forschungsgemeinschaft, by the European Research Council under ERC-Consolidator grant no.\ 614733, and by the Klaus-Tschira-Foundation.}

\maketitle

\section{Introduction}
Bill Thurston has had a tremendous impact on shaping the geometric imagination of many mathematicians. His new ways of thinking about geometric shapes changed the way we now think about hyperbolic surfaces, three-manifolds, and higher dimensional hyperbolic manifolds. 

There are many interesting geometries besides hyperbolic geometry, as for example complex hyperbolic geometry, projective geometry or other higher rank geometries, i.e. geometries whose transformation groups 
are semisimple Lie groups of higher rank, as for example 
$\SL(n,\RR)$, $n\geq 3$, or $\Sp(2n,\RR)$, $n\geq2$. 
In recent years several developments led to the discovery of 
interesting classes of discrete subgroups of Lie groups, which parametrize non-trivial deformation spaces of geometric structures, which are governed by higher rank Lie groups. 
One such development is the theory of Anosov representations of word hyperbolic groups which has been introduced by Fran\c{c}ois Labourie \cite{Labourie_anosov}.
The notion of Anosov representations is defined for representation of a finitely generated word hyperbolic group into any semisimple (even reductive) non-compact Lie group. 
However, in this article we restrict our attention to representations into the symplectic groups $\Sp(2n,\RR)$ and $\Sp(2n,\CC)$.  

The aim of this article is to describe, in a rather informal way and in a special case, some of the geometric aspects of Anosov representations, which lead us to think about manifolds modelled on higher rank 
geometries and try to build a geometric imagination for such geometric structures. 

We are dearly missing Bill Thurston's imagination. We will have to rely on and build our own. 
\section{Preliminaries}
\subsection{The symplectic group}
A symplectic vector space is a pair $(V_\KK, \omega_\KK)$, where $V_\KK$ is a $2n$-dimensional $\KK$-vector space, for  $\KK = \RR$ or $\CC$, and $\omega_\KK$ is a non-degenerate skew-symmetric bilinear form on $V$. 
The linear (real or complex) symplectic group of $(V_\KK,\omega_\KK)$ is 
$$
\Sp(V_\KK, \omega) = \{ g \in \GL(V_\KK)\, |\, \omega_\KK( gv, gw) = \omega_\KK(v,w) \text{ for all } v,w \in V_\KK\}
$$
If we choose a linear isomorphism $V_\KK \cong \KK^{2n}$ and represent $\omega_\KK$ by the matrix $ \begin{pmatrix}
  0 & \id \\
  -\id & 0
 \end{pmatrix}$, i.e.
$
\omega_\KK ( v, w) = v^T \begin{pmatrix}
  0 & \id \\
  -\id & 0
 \end{pmatrix} w,$ 
 then \[\Sp(V_\KK, \omega) \cong \Sp(2n,\KK) = \left\{g \in \GL(2n,\KK) \,|\,  g^T\begin{pmatrix}
  0 & \id \\
  -\id & 0
 \end{pmatrix} g = \begin{pmatrix}
  0 & \id \\
  -\id & 0
 \end{pmatrix}\right\}.\] 
 
We will sometimes denote $V_\RR$ simply by $V$ and $\omega_\RR$ by $\omega$. If $V_\CC$ is the complexification of $V$ and $\omega_\CC$ is the complex linear extension of $\omega$, we have a natural embedding of Lie groups 
\[\Sp(2n,\RR)\cong\Sp(V, \omega) \rightarrow \Sp(V_\CC, \omega) \cong \Sp(2n,\CC).\]

The symplectic group $\Sp(V_\KK, \omega_\KK)$ acts transitively on various homogeneous spaces: 
\begin{enumerate}
\item The Riemannian symmetric space $X_\RR(n)$ associated to $\Sp(2n,\RR)$, which can be identified with \[X_\RR(n) \cong\Sp(2n,\RR)/ \U(n).\]
\item The Riemannian symmetric space $X_\CC(n)$ associated to $\Sp(2n,\CC)$, which can be identified with \[X_\CC(n) \cong\Sp(2n,\CC)/ \Sp(n).\]
\item The affine symmetric spaces $X_{p,q}$,  $p+q = n$, which can be identified with  
\[X_{p,q}\cong \Sp(2n,\RR)/ \U(p,q).\] 
The extreme cases being $X_{0,n} = X_\RR (n) = X_{n,0}$. 

\item The space $\mathrm{Is}_i(V_\KK)$ of $\omega_\KK$-isotropic subspace of $V_\KK$ of dimension $i$. We can identify 
 \[ \mathrm{Is}_i(V_\KK) \cong \Sp(V_\KK, \omega_\KK)/ Q_1(V_\KK),\] where $Q_i(V_\KK)$ denotes the stabilizer of an isotropic i-dimensional subspace of $V_\KK$
Of particular interest to us will be the projective space $ \PP(V_\KK)$, which can be identified with  $\mathrm{Is}_1(V_\KK)$, and the 
the space of Lagrangians $ \mathrm{Lag}(V_\KK)$, which can be identified with $\mathrm{Is}_n(V_\KK)$. 

\end{enumerate}

\begin{remark}
\begin{enumerate}
\item 
Recall that a subspace 
$L\subset V_\KK$ is $\omega_\KK$-isotropic if $L \subset L^{\perp_{\omega_\KK}}$, where $L^{\perp_{\omega_\KK}}$ is the orthogonal complement with respect to $\omega_\KK$. Since $\omega_K$ is skew-symmetric, any one-dimensional subspace is isotropic. An isotropic subspace $L$ is maximal if it is of dimension $n$, or equivalently if $L = L^{\perp_{\omega_\KK}}$. In this case $L$ is called a Lagrangian subspace.
\item Two elements $F, F' \in \mathrm{Is}_i(V_\KK)$ are said to be {\em transverse} if $F \oplus F'^{\perp_{\omega_\KK}} = V_\KK = F' \oplus F^{\perp_{\omega_\KK}} $.
\item For every isotropic subspace  $Z \in \mathrm{Is}_i(V_\KK)$, the symplectic form $\omega_\KK$ induces a non-degenerate skew-symmetric bilinear form on the vector space $Z^{\perp_{\omega_\KK}}/ Z$, turning  $Z^{\perp_{\omega_\KK}}/ Z$ into a $2(n-i)$-dimensional symplectic vector space.

\end{enumerate}

\end{remark}

\subsection{Decomposition of the space of complex Lagrangian subspaces} 
Let $(V, \omega)$ be a symplectic vector space and  $(V_\CC, \omega_\CC)$ its complexification Then the space 
of complex Lagrangians $\mathrm{Lag}(V_\CC)$ decomposes into several $\Sp(V,\omega)$-orbits. 

First, there is a decomposition with respect to the real structure: for every $i = 0, \cdots, n$ we set 
$$\mathcal{R}_i = \{ W \, |\, dim(W \cap \overline{W} ) = i\},$$ where $\overline{W}$ is the complex conjugation with respect to $V \subset V_\CC$. 

\begin{proposition}\label{prop:orbits}
The set $\mathcal{R}_i$ fibers over the space $\mathrm{Is}_i(V)$ of $i$-dimensional  isotropic subspaces in $V$ with fiber isomorphic to $\bigcup_{p'+q' = n-i} X_{p',q'}$. 
The set $\mathcal{R}_n$ is the unique closed $\Sp(V,\omega)$ orbit and identifies with $\mathrm{Lag}(V)$.  The set $\mathcal{R}_0$ is the union of the $n+1$ open $\Sp(V,\omega)$-orbits, $\mathcal{R}_0 = \bigcup_{p=0, \cdots, n} H_{p,q}$, where $H_{p,q} \cong X_{p,q}$.  
\end{proposition}
\begin{proof}
We sketch the proof, leaving some verifications to the reader. A more detailed description of the embedding $X_\RR(n)$  into $\mathrm{Lag}(V_\CC)$ can be found for example in \cite{Satake_book}. 

Note that the imaginary part of $\omega_\CC$ gives rise to a non-degenerate Hermitian form $h$ of signature $(n,n)$ on $V_\CC$, 
\( h(v,w) = i \omega_\CC(\overline{v}, w) \)
This Hermitian form is preserved by $\Sp(V,\omega)$.  

Let us first describe $\mathcal{R}_0$. 
It is an easy calculation that if $W \in \mathcal{R}_0$, then the restriction of $h$ to $W\times W$ is non-degenerate and of signature $(p,q)$ with $p+q = n$. 
The set  $\mathrm{R}_0$ then decomposes into the disjoint union 
 \[\mathcal{R}_0 = \bigcup_{p=0, \cdots, n} \mathcal{H}_{p,q},\] 
 where 
 \[
 \mathcal{H}_{p,q} = \{ W \in\mathcal{R}_0\, |\, h|_{W\times W} \text{ is of signature } (p,q)\}.
 \]
On each $ \mathcal{H}_{p,q} $ the group $\Sp(V,\omega)$ acts transitively, and the stabilizer of $W\in \mathcal{H}_{p,q} $ is $U(W, h|_{W\times}) \cong U(p,q)$. Thus $ \mathcal{H}_{p,q} \cong X_{p,q}$.

If $W \in \mathcal{R}_i$, then $Z= W\cap \overline{W}$ is an $i$-dimensional (complex) isotropic subspace of $V_\CC$. Since $\overline{Z} = Z$, it is the complexification of a real $i$-dimensional isotropic subspace $Z'$ of $V$. This defines the projection $\mathcal{R}_i \rightarrow \mathrm{Is}_i(V)$.

 To describe the fiber consider $M=Z^{\perp_{\omega_\CC}}/Z$. The symplectic form $\omega_\CC$ induces a symplectic form on $M$, the real structure induces a real structure, and $h$ induces a non-degenerate Hermitian form of signature $(n-i, n-i)$. Note that any $W \in \mathcal{R}_i$ for which  $W\cap \overline{W} = Z$ is uniquely determined by a maximal isotropic (Lagrangian) subspace  $Y$ of $M$, which satisfies $Y\cap \overline{Y} = \{0\}$. Thus the description of $\mathcal{R}_0$ above  gives the description of the fiber. 
 In particular, $\mathrm{R}_i$ also decomposes into several $\Sp(V,\omega)$-orbits $H^i_{p', q'}$, where $p' + q' = n-i$ and $H^i_{p', q'}$ is the component of $\mathcal{R}_i$ with fiber isomorphic to $X_{p',q'}$. 

If $W \in \mathcal{R}_n$, then $W = \overline{W}$ and hence $W$ is the complexification of a real maximal $\omega$-isotropic subspace of $V$. The fiber of the projection $\mathcal{R}_n \to \mathrm{Lag}(V)$ is trivial.  
This gives the identification $\mathcal{R}_n \cong \mathrm{Lag}(V)$. In particular, $\mathcal{R}_n$ is a compact, and hence closed $\Sp(V,\omega)$-orbit. 
\end{proof}

Proposition~\ref{prop:orbits} gives explicit embeddings
\[X_{p,q} \hookrightarrow H_{p,q} \subset \mathrm{Lag}(V_\CC).\] Since $\mathrm{Lag}(V_\CC)$ is compact, the closure of $H_{p,q} \subset \mathrm{Lag}(V_\CC)$ gives rise to a compactification of the affine symmetric spaces $X_{p,q}$. The closure of $H_{p,q}$ has the following structure 
\[cl(H_{p,q}) = \bigcup_{p' \leq p, \, q' \leq q } H^i_{p',q'},\] with $H_{0,0} := \mathcal{R}_n$. 

For the Riemannian symmetric space $X_\RR(n)$  the embedding $X_\RR(n) \hookrightarrow \mathrm{Lag}(V_\CC)$ is called the Harish-Chandra-embedding. The described compactification is isomorphic to a (minimal) Satake-Furstenberg compactification, and to the bounded symmetric domain compactification of $X_\RR(n)$. 

\begin{remark}
\begin{enumerate}
\item 
In fact the bounded symmetric domain realization of $X$ is usually derived from the above embedding: Choosing two transverse Lagrangian subspaces $V_+$ and $V_-$ of $V_\CC$, such that $h$ is positive definite on $V_+$ and negative definite on $V_-$, and such that $V_+$ and $V_-$ are orthogonal with respect to $h$,  any other space $W \in H_{n,0}$ is transverse to $V_-$, and thus can be written as a graph of a linear map from $V_+ \rightarrow V_-$. The condition that $W$ is Lagrangian implies that this map is symmetric and the condition that $h$ is positive definite on $W$ translates into the condition that $Id - \overline{Z} Z$ is positive definite. For details we refer to \cite{Satake_book}.
\item 
When $n=1$ the space $ \mathrm{Lag}(V_\CC) $ is just the projective space $ \CC\PP^1$, which  decomposes into three $\SL(2,\RR)$-orbits: upper hemisphere, lower hemisphere, equator. 
\end{enumerate}
\end{remark}

\subsection{Symplectic Anosov representations} 
We shortly recall here the notion of Anosov representations. We focus on representations into the symplectic group which are Anosov with respect to the stabilizer of an isotropic subspace of $V_\KK$. 
 The definition we give relies on a characterization of Anosov representations given in \cite{GGKW_anosov}, which can also be obtained from \cite{KapovichLeebPorti, KapovichLeebPorti14}. For general discussions of Anosov representations we refer to \cite{Labourie_anosov, Guichard_Wienhard_DoD, GGKW_anosov, KapovichLeebPorti, KapovichLeebPorti14, KapovichLeebPorti14_2}. 

Let $\Gamma$ be a finitely generated word hyperbolic group. Its boundary $\partial_\infty \Gamma$ is a compact space, equipped with a natural action of $\Gamma$ as a uniform convergence group. In particular, every element $\gamma \in \Gamma$ of infinite order has a unique attracting fix point $\gamma^+ \in \partial_\infty \Gamma$ and a unique repelling fix point $\gamma^- \in  \partial_\infty \Gamma$ such that for every $x \in \partial_\infty \Gamma \backslash \{\gamma^-\}$ we have $\lim_{n\to \infty} \gamma^n \cdot x =  \gamma^+$. In the following we will assume for simplicity that $\Gamma$ is torsion free. 

The symmetric spaces $X_\KK$ of $\Sp(2n,\KK)$  are symmetric spaces of non-compact type of rank $n$. In such spaces the Riemannian distance function only partially determines the relative position of two points. More precise information is given by their distance vector. For group elements $g \in \Sp(2n,\KK)$ the distance vector between a fixed base point $x_0 \in X_\KK$ and $g \cdot x_0$ is measured by a Cartan decomposition with respect to $x_0$. 
To describe this in more detail, let $K$ be a maximal compact subgroup of $\Sp(2n,\KK)$, i.e. $K = \U(n)$ if $\KK = \RR$ and $K = \Sp(n)$ if $\KK = \CC$.  Let 
$\mathfrak{a} \subset \mathfrak{s}\mathfrak{p}(2,\KK)$ be the subset of diagonal matrices with real entries, i.e. 
\[\mathfrak{a} = \{ \mathrm{diag}(\lambda_1, \cdots \lambda_n, - \lambda_n, \cdots - \lambda_1) \, |\, \lambda_i \in \RR \}.\]
Then $\Sp(2n,\KK)$ admits a decomposition  $\Sp(2n,\KK) = K \exp({\mathfrak{a}})K$, i.e.  any element $g \in \Sp(2n,\KK) $ can be written as $g = k_1 \exp{(a_g)} k_2$, with $a_g \in \mathfrak{a}$ and  $k_1, k_2 \in K$. This decomposition is in general not unique. 
But if we consider the cone 
\[\mathfrak{a}^+ = \{ diag(\lambda_1, \cdots \lambda_n, - \lambda_n, \cdots - \lambda_1) \in \mathfrak{a}\, |\, \lambda_1\geq \lambda_2\geq \cdots \geq \lambda_n\geq 0 \} \subset \mathfrak{a},\] 
then  any element $g \in \Sp(2n,\KK) $ can be written as $g = k_1 \exp{(a_g)} k_2$, with an unique $a_g \in \mathfrak{a}^+$ and  $k_1, k_2 \in K$. This decomposition  
\[\Sp(2n,\KK) = K \exp({\mathfrak{a}^+})K\] is called a Cartan decomposition. 

It induces the {\em Cartan projection}
\[
\mu: \Sp(2n,\KK) \rightarrow \mathfrak{a}^+, \, g \mapsto \mu(g) := a_g. 
\]

The Cartan projection is a continuous, proper surjective map. 

We denote by $\alpha_i \mathfrak{a}^*$, $ i = 1, \cdots, n$ the linear map given by  $\alpha_i = \epsilon_i - \epsilon_{i+1}$, where $\epsilon_i$ is the linear form which sends a diagonal matrix to its $i$-th eigenvalue.  
Note that $\alpha_n = \epsilon_n - \epsilon_{n+1} = 2 \epsilon_n$.

\begin{definition}\label{def:anosov}
Let $\Gamma$ be a finitely generated word hyperbolic group and $\partial_\infty \Gamma$ its boundary.  A representation $\rho: \Gamma \to \Sp(V_\KK, \omega_\KK)$ is $Q_i$-Anosov, $i= 1, \cdots, n$ if the following holds
\begin{enumerate}
\item There exist a {\em continuous} $\rho$-equivariant map $\xi_i: \partial_\infty \Gamma \to \mathrm{Is}_i(V_\KK)$. 
\item The map $\xi_i$ is {\em dynamics preserving}, i.e. for every $\gamma \in \Gamma$ of infinite order, the images of the attracting resp. repelling fix point $\gamma^\pm$ under $\xi_i$ is the attracting resp. repelling fix point of $\rho(\gamma)$ in $\mathrm{Is}_i(V_\KK)$. 
\item The map $\xi_i$ is {\em transverse}, i.e. for every $t \neq t'$ in $\partial_\infty \Gamma$, the images under $\xi_i$ satisfy 
\[\xi_i(t) \oplus \xi_i(t')^{\perp_{\omega_\KK}} = V_\KK = \xi_i(t') \oplus \xi_i(t)^{\perp_{\omega_\KK}}.\] 
\item For every diverging sequence $\gamma_n$ in $\Gamma$ we have 
\[
\mathrm{lim}_{n\to \infty} \alpha_i(\mu(\gamma_n)) = \infty.
\]
\end{enumerate}
\end{definition}
\begin{remark}
\begin{enumerate}
\item 
The original definition \cite{Labourie_anosov, Guichard_Wienhard_DoD} uses the geodesic flow space of the hyperbolic group $\Gamma$. This definition is very useful in several applications, see for example \cite{Sambarino, Bridgeman_Canary_Labourie_Sambarino}.  Characterizations not involving the geodesic flow space are given in \cite{KapovichLeebPorti, KapovichLeebPorti14, KapovichLeebPorti14_2, GGKW_anosov}. Some of the characterizations in \cite{KapovichLeebPorti, KapovichLeebPorti14, KapovichLeebPorti14_2} do not require the group $\Gamma$ to be assumed to be word hyperbolic a priori.  
\item One can define $P$-Anosov representations with respect to any parabolic subgroup $P< \Sp(2n,\RR)$, but this notion essentially reduces to the notion of $Q_i$-Anosov representations above, because a representation is $P$-Anosov if it $Q_i$-Anosov for all maximal parabolic subgroups $Q_i$ which contain $P$.  For example, a representation $\rho: \Gamma \to \Sp(V_\KK, \omega_\KK)$ is $B$-Anosov, for $B<\Sp(2n,\KK)$ the minimal parabolic subgroups, if and only if it is $Q_i$-Anosov for every $i= 1,\cdots, n$.    
\item 
Under the embedding $\Sp(2n,\RR) \to \Sp(2n,\CC)$, a $Q_i$-Anosov representation $\rho: \pi_1(\Sigma) \to \Sp(2n,\RR)$ gives rise to a $Q_i$-Anosov representation $\rho_\CC:  \pi_1(\Sigma) \to \Sp(2n,\CC)$  into $\Sp(2n,\CC)$.
\end{enumerate}
\end{remark}

We recall, without proof several of the remarkable properties Anosov representations have. 
\begin{proposition}\label{prop:properties_individual}\cite{Labourie_anosov, Guichard_Wienhard_DoD}
Let $\rho: \Gamma \to \Sp(V_\KK, \omega_\KK)$ be a $Q_i$-Anosov representation. 
Then 
\begin{enumerate}
\item $\rho$ has finite kernel and discrete image. 
\item the map $\Gamma \rightarrow X_\KK(n), \, \gamma \mapsto \gamma x_0$, for $x_0\in X_\KK(n)$ fixed,  is a quasi-isometric embedding with respect to the word metric $|\cdot |_\Gamma$ on $\Gamma$ (with respect to a finite generating set) and the Riemannian distance $d_X$ on $X_\KK$, i.e. there exists constants $L,l>0$ such that for all $\gamma \in \Gamma$
\[
\frac{1}{L} |\gamma|_\Gamma - l \leq  d_X(\gamma x_0, x_0)\leq L |\gamma|_\Gamma + l, 
\]
where $x_0 \in X_\KK$ is a fixed base point. 
\item Every element of infinite order $\gamma \in \Gamma$ acts proximal on $\mathrm{Is}_i(V_\KK)$, i.e. there exists $F_\gamma^+$, $F_\gamma^-$ in $\mathrm{Is}_i(V_\KK)$ such that for every $F$ transverse to $F_\gamma^-$ the sequence 
$\rho(\gamma)^n F$ converges to $F_\gamma^+$, and for every $F'$ transverse to $F_\gamma^+$ the sequence  $\rho(\gamma^{-n}) F'$ converges to $F_\gamma^-$. In fact, for a $Q_i$-Anosov representation we have $F_\gamma^\pm = \xi_i(\gamma^\pm)$. 
\end{enumerate}
\end{proposition}

Also the set of Anosov representations inside the representation variety has remarkable properties. 
\begin{proposition}\label{prop:properties_set}\cite{Labourie_anosov, Guichard_Wienhard_DoD, Canary_survey, GGKW_anosov}
\begin{enumerate}
\item 
The set $\mathrm{Hom}_{Q_i-Anosov}(\Gamma, \Sp(V_\KK, \omega_\KK))$  of $Q_i$-Anosov representation is open in $\mathrm{Hom}(\Gamma, \Sp(V_\KK, \omega_\KK))$. 
\item The  group $\mathrm{Out}(\Gamma)$ of outer automorphisms acts properly discontinuously on  
$\mathrm{Hom}_{Q_i-Anosov}(\Gamma, \Sp(V_\KK, \omega_\KK))/\Sp(V_\KK, \omega_\KK)$, the set of conjugacy classes of Anosov representations. 
\item The set of Anosov representations $\mathrm{Hom}_{Q_i-Anosov}(\Gamma, \Sp(V_\KK, \omega_\KK))$ descends to a well defined subset of the character variety. 
\end{enumerate}
\end{proposition}

\subsection{Examples}
In this subsection we describe several examples of Anosov representations 

\subsubsection{Hitchin representations} 
Let $\pi_1(\Sigma)$ be the fundamental group of a closed surface of genus $g\geq 2$, and let $\iota: \pi_1(\Sigma) \to \mathrm{SL}(2,\RR)$ be a discrete embedding. We denote by $\pi_{2n}: \SL(2,\RR)  \to \Sp(2n,\RR)$ the $2n$-dimensional irreducible representation of $\SL(2,\RR)$. It is unique up to conjugation in $\Sp(2n,\RR)$.  We call the composition \[\rho_{irr}: \pi_{2n} \circ \iota: \pi_1(\Sigma) \to Sp(2n,\RR)\] an irreducible Fuchsian representation. A {\em Hitchin representation} is any representation $\rho: \pi_1(\Sigma) \to \Sp(2n,\RR)$ which can be continuously deformed to an irreducible Fuchsian representation. Hitchin representation were introduced by Hitchin in \cite{Hitchin}, who showed that there are $2\times2^{2g}$ connected components of Hitchin representations into $\Sp(2n,\RR)$ which are all homeomorphic to a ball of dimension $(2g-2)\dim{\Sp(2n,\RR)}$. The $2\times2^{2g}$ connected components come from the fact that there are $2^{2g}$ different ways of lifting a discrete embedding $\iota: \pi_1(\Sigma) \to \mathrm{PSL}(2,\RR)$ to $\SL(2,\RR)$ and $2$ distinct connected components of the space of discrete embeddings $ \pi_1(\Sigma) \to \mathrm{PSL}(2,\RR)$. 

Hitchin representations were extensively studied by Labourie in \cite{Labourie_anosov}, who in particular showed that Hitchin representations are $Q_i$-Anosov for all $i= 1, \cdots, n$. 

\subsubsection{Maximal representations}
The Toledo number is a characteristic number of a representation $\rho: \pi_1(\Sigma) \to \Sp(2n,\RR)$. It satisfies $|T(\rho)|\leq n |\chi(\Sigma)|$, where $\chi(\Sigma)$ is the Euler characteristic of $\Sigma$. Representations for which $T(\rho) = n |\chi(\Sigma)|$ are called {\em maximal representations}. Maximal representations are $Q_n$-Anosov, in particular they are discrete embeddings \cite{Goldman_88, Burger_Iozzi_Wienhard_toledo, Burger_Iozzi_Labourie_Wienhard}.

The set of maximal representations is a union of connected components, which contains the set of Hitchin representations as a proper subset. By definition any Hitchin representation can be deformed to a representation that factors through $\SL(2,\RR)$. This does not hold for maximal representations. In fact the set of maximal representations into $\Sp(4,\RR)$ contains $2g-4$ connected components in which every representation is Zariski dense. We refer the reader to \cite{Gothen, GarciaPrada_Mundet, GarciaPrada_Gothen_MundetPubli, Guichard_Wienhard_InvaMaxi, Bradlow_GarciaPrada_Gothen_SP4} for more details on the number of connected components and topological invariants of maximal representations into the symplectic group. 

Maximal representations are in general only $Q_n$, not $Q_1$-Anosov. Since some of the constructions we describe below are specifically for $Q_1$-Anosov representations, it is interesting to ask. 
\begin{question}
Are there maximal representations $\rho: \pi_1(\Sigma) \to \Sp(2n,\RR)$ outside of the Hitchin component that are $Q_1$-Anosov or even $B$-Anosov? What is the set of maximal representation that are also $Q_1$-Anosov?
\end{question}

\subsubsection{Deformations of $\SL(2,\RR)$ embeddings}
Examples of $Q_i$-Anosov representations that are not maximal (and in particular not Hitchin) can easily be constructed by considering other embeddings of $\SL(2,\RR)$.

For example, consider a symplectic vector space $(\RR^2, \omega_0)$ and let $(V, \omega) = (\RR^2 \oplus \RR^2, \omega_0\oplus \epsilon \omega_0)$, where $\epsilon\in \{-1, +1\}$, and $\iota: \pi_1(\Sigma) \to \mathrm{SL}(2,\RR)$ a discrete embedding.
\begin{enumerate}
\item The representation $(\iota, 1): \pi_1(\Sigma) \to \SL(2,\RR) \times \SL(2,\RR) \subset \Sp(V, \omega)$ is $Q_1$-Anosov, but not maximal. Since the set of $Q_1$-Anosov representations is open, the same holds for small deformations of this representation.  
\item For $\epsilon = -1$ the representation $(\iota, \iota): \pi_1(\Sigma) \to \SL(2,\RR) \times \SL(2,\RR) \subset \Sp(V, \omega)$ is $Q_n$-Anosov, but not maximal. They same holds true for small deformations. Note that for $\epsilon = 1$ the representation $(\iota, \iota): \pi_1(\Sigma) \to \SL(2,\RR) \times \SL(2,\RR) \subset \Sp(V, \omega)$ is maximal. 
\end{enumerate} 

\subsubsection{$\SL(2,\CC)$ embeddings and their deformations}
Consider a two-dimensional complex symplectic vector space $(V_0^\CC, \omega_0^\CC)$. Then $\Sp(V_0^\CC, \omega_0^\CC)  \cong \SL(2,\CC)$. 
If we forget the complex structure, we can identify $V_0^\CC = V$ with the real four-dimensional vector space $V$, and the real part of $\omega_0^\CC$ defines a symplectic form $\omega$ on $V$. 
This gives a natural embedding $\pi: \SL(2,\CC)  \to \Sp(4,\RR)$. 

Composing a quasi-Fuchsian representation $\iota: \pi_1(\Sigma) \to \SL(2,\CC)$ with $\pi$ we get a $Q_n$-Anosov representation $\rho: \pi_1(\Sigma) \to \SL(2,\CC) \to \Sp(4,\RR)$. 

This embedding also gives interesting examples when $\Gamma$ is the fundamental group of a closed hyperbolic $3$-manifold, and $ \iota: \Gamma \to \SL(2,\CC)$ is the embedding of a cocompact lattice. 
Then $\rho:=\pi \circ \iota: \Gamma \to \Sp(4,\RR)$ is a $Q_n$-Anosov representation. It was explained to me by Gye-Seon Lee, that under the local isomorphism between $\SL(2,\CC)$ and $\SO(1,3)$ , and between $\Sp(4,\RR)$ and $\SO(2,3)$, the above embedding $\SL(2,\CC) \hookrightarrow \Sp(4,\RR)$ corresponds to the natural embedding $\SO(1,3)  \hookrightarrow \SO(2,3)$. It is then  a consequence of the work of Barbot \cite{BarbotDefAdSQF} that any deformation of $\rho$ is $Q_n$-Anosov. 

\section{Domains of Discontinuity}
In this section we recall a construction of domains of discontinuity for $Q_i$-Anosov representations into $\Sp(2n,\KK)$, $i = 1$ or $n$,  which was introduced in \cite{Guichard_Wienhard_DoD}. 
We denote by $\hat{}$ the nontrivial involution of the set $\{1,n\}$, i.e. $\hat{1} = n$ and $\hat{n}= 1$. 

\subsection{Construction and statement}
Let $\rho: \Gamma \to \Sp(2n,\KK)$ be a $Q_i$-Anosov representation, $ i = 1$ or $n$ and $\xi_i: \partial_\infty \Gamma \to \mathrm{Is}_i(V_\KK)$ its continuous $\rho$-equivariant boundary map. 
We are going to use $\xi_i$ to exhibit a set $\mathcal{K}_{\xi_i} \subset \mathrm{Is}_{\hat{i}} (V_\KK)$, such that the action of $\Gamma$  (via $\rho$) 
 on the complement of $\mathcal{K}_{\xi_i} $ in $\mathrm{Is}_{\hat{i}} (V_\KK)$ is proper. 
 
 Given an (isotropic) line $l \in \PP(V_\KK)$, we set
  \[
 K_l:= \{ L\in \mathrm{Lag}(V_\KK) \, |\, l \subset L\}.  
 \]
 In words, for a line $l$, $K_l$ is the set of Lagrangian subspaces containing $l$. Note that we can identify $K_l$ with the set $ \mathrm{Lag} (l^{\perp_{\omega_\KK}}/l)$. 
 
 Given a Lagrangian subspace $L \in \mathrm{Lag}(V_\KK)$ we set 
  \[
 K_L:= \{ l\in \PP(V_\KK)\, |\, l \subset L\}.  
 \]
 In words, for a Lagrangian subspace $L$, $K_L$ is the set of lines contained in $L$. Note that we can identify $K_L$ with $\PP (L)$. 

\begin{proposition}\cite[Proposition~8.1.]{Guichard_Wienhard_DoD}\label{thm:dod}
Let $\rho: \Gamma \to \Sp(2n,\KK)$ be a $Q_i$-Anosov representation, $ i =1$ or $n$ and $\xi_i: \partial_\infty \Gamma \to \mathrm{Is}_i(V_\KK)$ its continuous $\rho$-equivariant boundary map.
Set 
\[\mathcal{K}_{\xi_i} := \bigcup_{t\in \partial_\infty \Gamma} K_{\xi_i(t)} \subset \mathrm{Is}_{\hat{i}} (V_\KK).\]

The complement 
\[\Omega_{\xi_i} = \mathrm{Is}_{\hat{i}} (V_\KK)\backslash\mathcal{K}_{\xi_i}\] is a $\Gamma$-invariant open subset. The action of $\Gamma$ (via $\rho$) on  $\Omega_{\xi_i}$ is properly discontinuous and cocompact. 

\end{proposition}

\begin{remark}
\begin{enumerate}
\item Note that $\mathcal{K}_{\xi_i}$ is in fact the {\em disjoint} union of $ K_{\xi_i(t)}$, $t\in \partial_\infty \Gamma$. 
\item 
In general the domain of discontinuity $\Omega_{\xi_i}$ could be empty, and in fact this happens for example if $\Gamma = \pi_1(\Sigma)$ is a surface group, $n = 1$, and $\KK = \RR$, or if $\Gamma$ is the fundamental group of a closed hyperbolic $3$-manifold, $n= 1$, and $\KK = \CC$. 
When $\Gamma$ is a free group, then $\Omega_{\xi_i}$ is always nonempty. In general,  $\Omega_{\xi_i}$ is nonempty, as soon as $n$ is bigger than the dimension of $\partial_\infty \Gamma$. 
\end{enumerate}
\end{remark}

We refer the reader to \cite{Guichard_Wienhard_DoD} for a proof of Theorem~\ref{thm:dod}. Here we would like to illustrate how the construction of the domains of discontinuity can be used to deduce further information on proper actions on homogeneous space and compactifications of the corresponding quotient manifolds. More results in this direction are proved in \cite{Guichard_Wienhard_DoD,GGKW_anosov, GGKW_compactif}.

\subsection{Complexifying $Q_1$-Anosov representations}\label{sec:Q0}
Let $\rho: \Gamma \to \Sp(V, \omega)$ be a $Q_1$-Anosov representation with boundary map $\xi: \partial_\infty \Gamma \to \PP(V)$. Let $(V_\CC, \omega_\CC)$ be the complexification of $(V,\omega)$ and $\rho_\CC: \Gamma \to \Sp(V_\CC, \omega_\CC)$ the induced representation, which is also $Q_1$-Anosov. We denote by $\xi_\CC: \partial_\infty \Gamma \to \PP(V_\CC)$ its boundary map. 
Since $\xi_\CC$ is the complexification of $\xi$ we have 
\begin{equation}\label{eq:cc}
\overline{\xi_\CC(t)} = \xi_\CC(t)
\end{equation} for all $t \in \partial_\infty \Gamma$. 

By Theorem~\ref{thm:dod} we obtain a domain of discontinuity $\Omega_{\xi_\CC} \subset \mathrm{Lag}(V_\CC)$ on which $\Gamma$ acts (via $\rho_\CC$) properly discontinuous and cocompact. 
Let us denote the compact complex quotient manifold $\Gamma \backslash \Omega_{\xi_\CC}$ by $M_{\rho_\CC}$. 

\begin{question}\label{question:M}
What is the homeomorphism type of $M_{\rho_\CC}$? What are the properties of $M_{\rho_\CC}$ as complex manifold? 
\end{question}

In order to describe the relation of $\Omega_{\xi_\CC} $ with the $\Sp(V,\omega)$-orbits described above, let us analyze in more detail the structure of the set $\mathcal{K}_{\xi_\CC}$. By definition 
\[\mathcal{K}_{\xi_\CC} = \bigcup_{t \in \partial_\infty\Gamma} K_{\xi_\CC(t)} =  \bigcup_{t \in \partial_\infty\Gamma} \{ L \in  \mathrm{Lag}(V_\CC) \, |\, \xi_\CC(t) \subset L\}.
\]

\begin{lemma}\label{lem:R0}
The domain of discontinuity $\Omega_{\xi_\CC} $ contains $\mathcal{R}_0$. 
\end{lemma}

\begin{proof}
Let $L\in \mathcal{K}_{\xi_\CC}$, then $L$ contains $\xi_\CC(t)$ for some $t\in \partial_\infty \Gamma$. By \eqref{eq:cc}, this implies that $\xi_\CC(t) \subset \overline{L} \cap L$, and in particular $L \notin \mathcal{R}_0$, hence $\mathcal{R}_0 \subset\Omega_{\xi_\CC} $
\end{proof} 

\begin{corollary}\label{cor:prop}
Let  $ \rho: \Gamma \to \Sp(V, \omega)$ be a $Q_1$-Anosov representation. Then $\Gamma$ acts (via $\rho$) properly discontinuous on $X_{p,q}$. 
\end{corollary}

Analogously to Question~\ref{question:M} one can ask 
\begin{question}\label{question:X}
What is the homeomorphism type of $N^{p,q}_\rho = \Gamma \backslash X_{p,q}$? What are the properties of $N^{p,q}_\rho$ as complex manifold? 
\end{question}

Note that the non-compact manifolds $N^{p,q}_\rho = \Gamma \backslash X_{p,q}$ embed into $M_{\rho_\CC}$, and since $M_{\rho_\CC}$ is compact, the closure of $N^{p,q}_\rho$ in $M_{\rho_\CC}$ provides a compactification of the locally (affine) symmetric manifold $\Gamma \backslash X_{p,q}$. 

Let us describe in a bit more detail this compactification for the locally Riemannian symmetric space $\Gamma \backslash X_\RR(n)$.  Since the compactification of $N_\rho = \Gamma \backslash X(n)$ obtained by taking the closure in $M_{\rho_\CC}$ is the quotient of  $cl(H_{n,0}) \cap \Omega_{\xi_\CC} $ by the action of $\Gamma$ (via $\rho$), we describe the set $cl(H_{n,0}) \cap \Omega_{\xi_\CC} $. 

The compactification of $X_\RR(n)$ is given by  
\[cl(H_{n,0}) = \bigcup_{n' \leq n} H^{i}_{n',0},\] where $i = n-n'$. 
The set 
$H^{i}_{n',0}$ is a fiber bundle over the space  $\mathrm{Is}_i(V)$ of $i$-dimensional isotropic subspaces of $V$ with fiber isomorphic to $X_\RR(n')$.

Given $l \in \PP(V)$ we set 
\[S^i(l) = \{ Z \in \mathrm{Is}_i(V) \, |\, l \subset Z\}.\] 

\begin{lemma}
The intersection $\mathcal{K}_{\xi_\CC}\cap H^{i}_{n',0}$ is the restriction of the bundle $H^{i}_{n',0}$ to the subset $S^i_\xi := \bigcup_{t\in \partial_\infty \Gamma} S^i(\xi(t))$. 
\end{lemma}
\begin{proof}
Consider a complex Lagrangian $W \in H^{i}_{n',0}$ and let $Z' \in\mathrm{Is}_i(V)$ be the real $i$-dimensional isotropic subspace such that $Z'_\CC = W \cap \overline{W}$. Then $W$  is in $\mathcal{K}_{\xi_\CC}$ if and only if there exists $t\in \partial_\infty \Gamma$  with $\xi_\CC(t) \subset W\cap \overline{W} = Z'_\CC$. This happens by \eqref{eq:cc} if and only if $\xi(t) \in Z'$, i.e. $Z' \in S^i(\xi(t))$.
\end{proof}

\subsubsection{Deforming real $Q_1$-Anosov representation}
In the previous section we analysed the structure of the domain of discontinuity of a $Q_1$-Anosov representation $\rho': \Gamma \to \Sp(V_\CC, \omega_\CC)$ with boundary map $\xi': \partial_\infty \Gamma \to \PP(V_\CC)$ 
in the special case, when $\rho' = \rho_\CC$  is the complexification of a real $Q_1$-Anosov representation $\rho: \Gamma \to \Sp(2n,\RR)$. When we deform $\rho_\CC$ to a $Q_1$-Anosov representation $\rho': \Gamma \to \Sp(V_\CC, \omega_\CC)$ which does not preserve the real structure, the domains of discontinuity and the quotient manifolds 
$M_{\rho_\CC}$ and $M_{\rho'}$ are homeomorphic \cite[Theorem 9.12]{Guichard_Wienhard_DoD}, but $M_{\rho'}$ is not anymore invariant under complex conjugation, and there is no equivariant decomposition of $M_{\rho'}$ with respect to the real structure. 

The problem of gaining a better understanding of $M_{\rho'}$ is of particular interest when $\Gamma = \pi_1(\Sigma)$ is a surface group.  

In the case when $\mathrm{dim}(V) = 2$, $Q_1$-Anosov representations of $\pi_1(\Sigma) \to \SL(2,\RR) \hookrightarrow \SL(2,\CC)$ are discrete embeddings, 
i.e. a Fuchsian representations. The set $\mathcal{K}_{\xi_\CC}$ we remove is $\RR\PP^1 \subset \CC\PP^1$, and the domain of discontinuity $\Omega_{\xi_\CC} = \Omega^+ \cup \Omega^-$ is the union of the upper and the lower hemisphere. The  quotient manifold $M_{\rho_\CC}$ is the union of two Riemann surfaces which are homeomorphic to $\Sigma$ and carry - up to an orientation reversing isometry the same conformal structure.  
A  $Q_1$-Anosov representation $\rho':\Gamma \to \SL(2,\CC)$  is precisely a quasi-Fuchsian representation of $\pi_1(\Sigma)$. In this case $\mathcal{K}_{\xi_\CC} \subset \CC\PP^1$ is a quasi-circle. 
The quotient manifold $M_{\rho' }$ is still the union of two Riemann surfaces, now with different conformal structures. 

In this case, $n=1$, the Ahlfors-Bers simultaneous uniformization theorem, gives a parametrisation of the space of $Q_1$-Anosov representations of $\pi_1(\Sigma) \to \SL(2,\CC)$ by two copies of the Teichm\"uller space of $\Sigma$. 
It is an intriguing question to ask whether an analogue of the simultaneous uniformization theorem holds for deformations of Hitchin representations when $\mathrm{dim}(V) > 2$. 
A question along these lines was first raised by Fock and Goncharov \cite[Section~11]{Fock_Goncharov}. 

In order to make this question more precise we define

\begin{definition}\label{defi:QH}
The set of {\em Quasi-Hitchin representations} is the set of representation $\rho': \pi_1(\Sigma) \to \Sp(2n,\CC)$ which are $B$-Anosov with respect to a full isotropic flag, and which can be deformed to a  
representation \[\rho_\CC: \pi_1(\Sigma) \to \Sp(2n,\RR) \hookrightarrow \Sp(2n,\CC),\] which is the complexification of a Hitchin representation into $\Sp(2n,\RR)$. 

The set of {\em $Q_1$-Quasi-Hitchin representations} is the set of representation $\rho': \pi_1(\Sigma) \to \Sp(2n,\CC)$ which are $Q_1$-Anosov, and which can be deformed to a  
representation $\rho_\CC: \pi_1(\Sigma) \to \Sp(2n,\RR) \hookrightarrow \Sp(2n,\CC)$, which is the complexification of a Hitchin representation into $\Sp(2n,\RR)$. 
\end{definition}

\begin{question}
Is there a meaningful generalisation of the simultaneous uniformization theorem to Hitchin representations into $\Sp(2n,\RR)$? 
Is there a parametrisation of the set of Quasi-Hitchin representations (or of the set of $Q_1$-Quasi-Hitchin representations) in terms of Hitchin components for $\Sp(2n,\RR)$?
\end{question}

Note that the domain of discontinuity $\Omega_{\xi_\CC}$ always contains both the (upper) Siegel space $H_{n,0}$, as well as the (lower) Siegel space $H_{0,n}$,  which are mapped to each other under complex conjugation. When $n=1$ these are just the upper and the lower hemisphere. However, in the general case   $\Omega_{\xi_\CC}$ is connected, and the various open quotient manifolds $\Gamma \backslash H_{p,q}$ are glued along their boundaries in $M_{\rho_\CC}$. 

\begin{question}
How can one exhibit deformations of these open quotient manifolds $\Gamma \backslash H_{p,q}$ in $M_{\rho'}$ for a complex deformation $\rho'$ of $\rho_\CC$.
\end{question}

\subsection{Products} 
Let us describe another example, where the construction of the domain of discontinuity for $Q_1$-Anosov representation leads to some interesting consequences. We refer to \cite{GGKW_anosov, GGKW_compactif} for more details and proofs. 

Consider the symplectic vector space $(V_\KK, \omega_\KK)$. Let \[(V', \omega') = (V_\KK \oplus V_\KK, \omega_\KK \oplus -\omega_\KK),\] and let 
\[ \Sp(V_\KK, \omega_\KK) \times \Sp(V_\KK, -\omega_\KK)  \hookrightarrow  \Sp(V', \omega')\] be the corresponding embedding. 

Assume that  $\rho_0 : \Gamma \to \Sp(V_\KK,\omega_\KK)$ is a $Q_1$-Anosov with boundary map $\xi_0: \partial_\infty \Gamma \to \PP(V_\KK)$ and consider the representation 

\[\rho = (\rho_0, 1) : \Gamma \to \Sp(V_\KK, \omega_\KK) \times \Sp(V_\KK, - \omega_\KK)  \hookrightarrow \Sp(V', \omega'),\] 
where $1$ is the trivial representation. Then $\rho$ is $Q_1$-Anosov as a representation into $\Sp(V', \omega')$ with boundary map $\xi:\partial_\infty \Gamma \to \PP(V')$ given by 
\[t \mapsto \xi_0(t) \oplus 0 \in \PP(V_\KK \oplus 0) \subset \PP (V').\]

So the above construction provides a domain of discontinuity $\Omega_\xi \subset \mathrm{Lag}(V')$, obtained as the complement of 
\[
\mathcal{K}_\xi = \bigcup_{t\in \partial_\infty \Gamma} K_\xi(t)  = \bigcup_{t\in \partial_\infty \Gamma}\{ W \in \mathrm{Lag}(V')\, |\, \xi(t) \subset W\}.
\]

With respect to the symplectic splitting $V' = V_\KK \oplus V_\KK$, the space  $\mathrm{Lag}(V')$ decomposes into several $\Sp(V_\KK, \omega_\KK) \times \Sp(V_\KK, - \omega_\KK)$. 

The orbits are 
\begin{equation}\label{eq:orbits}
\mathcal{T}_i := \{ W \in \mathrm{Lag}(V')\, |\, \mathrm{dim}(W \cap V_\KK\oplus 0) = i\}.
\end{equation} 
The open orbit is $\mathcal{T}_0$. It contains in particular the diagonal 
\[\Delta(V_\KK) = \{ (v,v) \in V'\, |\, v \in V_\KK\},\] 
whose $\Sp(V_\KK, \omega_\KK) \times \Sp(V_\KK, - \omega_\KK)$-orbit identifies with 
\[\Sp(V_\KK, \omega_\KK) \times \Sp(V_\KK, - \omega_\KK)/ \Delta(\Sp(V_\KK, \omega_\KK)) \cong \Sp(V_\KK, \omega_\KK),\] 
where $\Delta: \Sp(V_\KK, \omega_KK) \to \Sp(V_\KK, \omega_\KK) \times \Sp(V_\KK, - \omega_\KK)$ denotes the diagonal embedding. 

\begin{proposition}
The domain of discontinuity $\Omega_\xi \subset \mathrm{Lag}(V')$ contains $\mathcal{T}_0$ as an open and dense submanifold. The quotient manifold $N_\rho= \Gamma \backslash \mathcal{T}_0$, which identifies with $\Gamma\backslash \Sp(V_\KK, \omega_\KK)$, embeds into $M_\rho = \Gamma\backslash \Omega_\xi$ as an open and dense submanifold. In particular, the manifold $M_\rho$ provides a compactification of $ N_\rho=\Gamma\backslash \Sp(V_\KK, \omega_\KK)$. 
\end{proposition}
%

\subsection{$Q_n$-Anosov representations}
Let $\rho: \Gamma \to \Sp(2n,\RR)$ be a $Q_n$-Anosov representation with boundary map $\xi: \partial_\infty \Gamma \to \mathrm{Lag}(\RR^{2n})$. 
Then the construction of the domain of discontinuity above provides a domain of discontinuity $\Omega_\xi \subset \PP(\RR^{2n})$. In particular, the quotient manifold 
 $M_\rho: \Gamma \backslash \Omega_\xi$ is endowed with a projective structure. For Hitchin representations $\rho: \pi_1(\Sigma) \to \Sp(4,\RR)$ this domain of discontinuity was considered in detail in \cite{Guichard_Wienhard_Duke}.
 
%

It is however also of interest to consider subsets in $\mathrm{Lag}(\RR^{2n})$ on which the image of a $Q_n$-Anosov representation acts properly. 

Let us consider the 2n-dimensional symplectic vector space $(V_0, \omega_0)$ and set $(V, \omega) = (V_0 \oplus V_0, \omega_0 \oplus -\omega_0)$. Let 
$\Sp(V_0, \omega_0) \times \Sp(V_0, -\omega_0) \hookrightarrow  \Sp(V, \omega)$ be the corresponding embedding. 

Assume that  $\rho_1, \rho_2 : \Gamma \to \Sp(V_0,\omega_0)\cong \Sp(2n,\RR)$ are $Q_n$-Anosov representations, 
 then the representation 
\[\rho = (\rho_1, \rho_2): \Gamma \to \Sp(V_0,\omega_0)\times \Sp(V_0, -\omega_0) \hookrightarrow  \Sp(V, \omega)\cong \Sp(4n,\RR)\]
 is $Q_{2n}$-Anosov. 
Let $\xi: \partial_\infty \Gamma \to \mathrm{Lag}(V)$ be the corresponding boundary map. 
Instead of the domain of discontinuity $\Omega_\xi \subset \PP(V)$ constructed above, let us consider the set $\mathcal{D}_\xi \subset \mathrm{Lag}(V)$ defined as follows. 

For a Lagrangian $L \in \mathrm{Lag}(V)$  we consider \[N_L = \{ L' \in \mathrm{Lag}(V)\, |\, L' \cap L \neq \{ 0\}\},\] the set of all Lagrangians which are not transverse to $L$, and set 
\[
\mathcal{N}_\xi = \bigcup_{t\in \partial_\infty \Gamma} N_\xi(t). 
\]
Note that here $N_\xi(t)$ and $N_\xi(t')$ for $t \neq t'$ will in general intersect, and the union is {\em not} disjoint. 

We set $\mathcal{D}_\xi $ to be the complement of $\mathcal{N}_\xi$ in $\mathrm{Lag}(V)$. 
The action of $\rho$ on $\mathcal{D}_\xi $ is properly discontinuous, but will not be cocompact. 

When $n=1$, then the Toledo number of the representation $\rho = (\rho_1, \rho_2): \Gamma \to \Sp(4,\RR)$ is zero. 
In that case$\mathcal{T}_0 \cong \SL(2,\RR)$ can be identified with Anti-de Sitter space, and the set $\mathcal{D}_\xi \cap \mathcal{T}_0\subset \mathrm{Lag}(V)$ is described in detail in \cite{BarbotDefAdSQF}, where also  its relation to the domain of discontinuity in Anti-de Sitter space defined by Mess in \cite{Mess} is described. 

\begin{question}
Develop a description of the set $\mathcal{D}_\xi \cap \mathcal{T}_0\subset \mathrm{Lag}(V)$ when $n\geq 2$. 
\end{question}

\section{Some questions} 
Let us mention a few other questions concerning Anosov representations. 

\subsection{Connected components} 
There are several examples known, where entire connected components of the space of homomorphisms $\mathrm{Hom}(\Gamma, G)$ consist of Anosov representations. In the case when $\Gamma$ is a surface group 
this happens for example for the set of Hitchin representations or the set of maximal representations. But this phenomena also occurs when $\Gamma$ is the fundamental group of a closed hyperbolic manifold of dimension $n$. In this case connected component which exist entirely of Anosov representations have been described by Benoist \cite{Benoist_CD1} and by \cite{BarbotDefAdSQF, Barbot_Merigot_fusionPubli}. 

\begin{question}
For which finitely generated word hyperbolic groups $\Gamma$ and for which non compact simple Lie groups $G$ are there entire connected components in $\mathrm{Hom}(\Gamma, G)$ which consist entirely of Anosov representations? 
Is there a common underlying structure that makes this possible? 
\end{question}

In some cases the set of Anosov representation consists of essentially one connected component, for example for $Q_1$-Anosov representations into $\SL(2,\KK)$, or for the Anosov representations considered in \cite{BarbotDefAdSQF}. 
But in general the set of Anosov representations has several connected components, which contain representations of rather different geometric flavour. 
\begin{question}
Investigate the set of connected components of the space of Anosov representations. 
\end{question}

\subsection{Quasi-isometry invariance} 
Recently, Haissinsky \cite{Haissinsky} proved that the class of convex cocompact Kleinian groups is quasi-isometrically rigid, i.e. if $\Gamma$ is a finitely generated group, which is quasi-isometric to a discrete convex cocompact subgroup $\Lambda < \PSL(2,\CC)$, then there exists a finite index subgroup $\Gamma'$ of $\Gamma$ and an isomorphism of $\Gamma'$ onto a discrete convex cocompact subgroup in $\PSL(2,\CC)$. 

For $\PSL(2,\CC)$ and $\Gamma$ a finitely generated hyperbolic group, a representations $\rho: \Gamma \to \PSL(2,\CC)$ is Anosov if and only if $\rho$ has finite kernel and the image of $\rho$ is a discrete convex cocompact subgroup in $\PSL(2,\CC)$. 
The following questions are thus a natural generalisation of Haissinsky's theorem
\begin{question}
Is the class of finitely generated groups that admit a $Q_i$-Anosov representation into $\Sp(2n,\KK)$ quasi-isometrically rigid? 
Is the class of finitely generated groups that admit any Anosov representation into a non compact simple Lie group quasi-isometrically rigid?

\end{question}

%
%
%

\def\cprime{$'$}
\providecommand{\bysame}{\leavevmode\hbox to3em{\hrulefill}\thinspace}
\providecommand{\MR}{\relax\ifhmode\unskip\space\fi MR }
\providecommand{\MRhref}[2]{%
  \href{http://www.ams.org/mathscinet-getitem?mr=#1}{#2}
}
\providecommand{\href}[2]{#2}

\end{document}